\documentclass[12pt]{article}

\usepackage[a4paper,margin=1in]{geometry}
\usepackage{amsmath,amssymb,amsthm,mathtools}
\usepackage{enumitem}
\usepackage{booktabs}
\usepackage[colorlinks=true,allcolors=blue]{hyperref}
\usepackage{microtype}

\hypersetup{
  pdftitle={Late-Time Fractional-Order Identification in Diffusion},
  pdfauthor={Niyaz Tokmagambetov},
  pdfkeywords={Caputo derivative, fractional diffusion, inverse problem, Mittag-Leffler function, order identification, resolvent moment, stability, sharpness, counterexample}
}

\theoremstyle{plain}
\newtheorem{theorem}{Theorem}[section]
\newtheorem{lemma}[theorem]{Lemma}
\newtheorem{proposition}[theorem]{Proposition}
\newtheorem{corollary}[theorem]{Corollary}
\theoremstyle{definition}
\newtheorem{assumption}[theorem]{Assumption}
\newtheorem{example}[theorem]{Example}
\theoremstyle{remark}
\newtheorem{remark}[theorem]{Remark}

\newcommand{\R}{\mathbb{R}}
\newcommand{\N}{\mathbb{N}}
\newcommand{\Dom}{\operatorname{Dom}}
\newcommand{\dd}{\,\mathrm{d}}
\newcommand{\cA}{\mathcal{A}}

\newcommand{\inner}[2]{\langle #1,#2\rangle}
\newcommand{\norm}[1]{\left\lVert #1\right\rVert}
\newcommand{\abs}[1]{\left\lvert #1\right\rvert}

\title{Late-Time Fractional-Order Identification in Caputo Diffusion Equation}

\author{\normalsize Niyaz Tokmagambetov\\[2mm]
\small Institute of Mathematics and Mathematical Modeling\\
\small 28 Shevchenko Street, 050010 Almaty, Kazakhstan\\
\small \texttt{tokmagambetov@math.kz}\\
\small \texttt{niyaz.tokmagambetov@gmail.com}}

\date{}

\begin{document}
\maketitle

\begin{abstract}
We study late-time identification of the Caputo order in a linear diffusion equation generated by a strictly positive self-adjoint operator with compact resolvent. For signed scalar observations
\(M_\alpha(t)=\sum_n a_nE_{\alpha,1}(-\lambda_nt^\alpha)\) satisfying \(\sum_n|a_n|/\lambda_n<\infty\), we show that, after eigenspace grouping, every nontrivial observation has a finite first nonzero resolvent moment \(S_m=\sum_n a_n/\lambda_n^m\). A uniform differentiated large-argument expansion of the Mittag-Leffler factor yields eventual strict monotonicity of \(\alpha\mapsto M_\alpha(t)\) on admissible intervals avoiding the zeros of \(1/\Gamma(1-m\alpha)\), hence uniqueness from one sufficiently late scalar measurement. For two measurements,
\(M_\alpha(\rho t)/M_\alpha(t)=\rho^{-m\alpha}(1+O(t^{-\alpha_0}))\), giving a log-ratio estimator with asymptotic-bias and relative-noise error bounds. For bounded observations, \(S_m=\langle\mathcal A^{-m}\varphi,h\rangle\); for a finite rod, the leading point-sensor condition is \((\mathcal A^{-1}\varphi)(x_*)\ne0\). Counterexamples show the sharpness of the exclusions and noise interpretation.
\end{abstract}



\medskip
\noindent\textbf{Funding:} This research was funded by the Science Committee of the Ministry of Science and Higher Education of the Republic of Kazakhstan, Grant No. AP23487589.

\bigskip
\section{Introduction}

Time-fractional diffusion equations are common macroscopic models for anomalous transport, viscoelastic relaxation, and diffusion in media with memory. Fractional time derivatives entered continuum modeling through Caputo's work on nearly frequency-independent dissipation and later became standard in fractional diffusion and random-walk descriptions of anomalous transport; see \cite{Caputo1967,Wyss1986,SchneiderWyss1989,MetzlerKlafter2000,Mainardi2010}. In the simplest subdiffusive model the first-order time derivative in the heat equation is replaced by the Caputo derivative of order \(\alpha\in(0,1)\). For an absolutely continuous scalar- or Hilbert-space-valued function \(u\), it is given by
\begin{equation}\label{eq:caputo}
        \partial_t^\alpha u(t)
        =\frac{1}{\Gamma(1-\alpha)}\int_0^t (t-s)^{-\alpha}u'(s)\dd s.
\end{equation}
The forward solutions used below are mild solutions defined by spectral functional calculus; formula \eqref{eq:caputo} applies whenever the corresponding time regularity is available. Throughout the analytical sections, \(t\) is time measured relative to a fixed reference scale, so logarithms of time are dimensionless. The order \(\alpha\) determines the strength of memory and the algebraic relaxation rate. It is therefore a natural unknown parameter in calibration problems for fractional diffusion models.

The direct theory of fractional evolution equations is well developed; see, for example, \cite{SamkoKilbasMarichev1993,Podlubny1999,KilbasSrivastavaTrujillo2006,GorenfloKilbasMainardiRogosin2020,EidelmanKochubei2004,SakamotoYamamoto2011,GorenfloLuchkoYamamoto2015}. Inverse problems for fractional partial differential equations include identification of orders, potentials, source terms, and other coefficients; see \cite{JinRundell2015,LiYamamoto2015,LiImanuvilovYamamoto2016,LiHuangYamamoto2020,KianLiLiuYamamoto2021,KaltenbacherRundell2019,KaltenbacherRundell2023,LiuLiYamamoto2019Sources,LiLiuYamamoto2019Parameters}. The order-identification problem is analytically distinctive because the unknown parameter appears both in the memory kernel and in the special functions representing the solution. Classical results on complete monotonicity and integral representations of the Mittag-Leffler function provide the background; see \cite{Pollard1948,GorenfloMainardi1997,Mainardi2014,GorenfloLoutchkoLuchko2002}.

Late-time order recovery itself is established in earlier work. Hatano, Nakagawa, Wang, and Yamamoto \cite{HatanoNakagawaWangYamamoto2013} derived reconstruction formulas from the time history at one fixed spatial point using small- and large-time asymptotics. In particular, their large-time formula identifies the order from a logarithmic decay slope. Ashurov and Umarov \cite{AshurovUmarov2020} proved uniqueness from the solution value at a fixed time and monitoring location for a subdiffusion equation with a general second-order elliptic operator. Li, Wang, Jia, and Zhang \cite{LiWangJiaZhang2023} likewise studied a single space-time measurement, reduced the inverse problem to a nonlinear algebraic equation, and used strict monotonicity. Therefore neither one-point uniqueness nor the fact that the fractional order is encoded in an algebraic late-time decay is claimed as new here. When the first moment is nonzero, the log-ratio estimator below is a finite-difference analogue of the large-time logarithmic-slope formula in \cite{HatanoNakagawaWangYamamoto2013}.

Parameter differentiation of Mittag-Leffler-type functions has also been studied directly. Rogosin, Giraldi, and Mainardi \cite{RogosinGiraldiMainardi2025} justified differentiation with respect to parameters by uniform-convergence and Mellin--Barnes arguments. Recent works prove additional monotonicity properties in the fractional parameter; see Sulaymonov and Ashurov \cite{SulaymonovAshurov2025} and Ferreira and Simon \cite{FerreiraSimon2025}. The differentiated expansion proved below has a different purpose: it gives a uniform large-argument remainder for the composite map \(\alpha\mapsto E_{\alpha,1}(-\lambda t^\alpha)\), with bounds strong enough to be summed over a full spectral observation.

The specific contribution of this paper is the resolvent-moment mechanism for signed scalar observations. It permits cancellation between modes, identifies the first surviving power of \(t^{-\alpha}\), isolates the exceptional reciprocal-gamma orders, and yields a two-time relative-error estimate for an arbitrary cancellation index \(m\). A moment-completeness lemma shows that this index is finite for every nontrivial observation satisfying the summability assumption. Thus the novelty lies in the signed-observation and cancellation hierarchy, not in the basic late-time decay principle itself.

Consider the abstract homogeneous problem
\begin{equation}\label{eq:intro-forward}
        \partial_t^\alpha u_\alpha(t)+\cA u_\alpha(t)=0,
        \qquad u_\alpha(0)=\varphi,
\end{equation}
where \(\cA\) is positive self-adjoint with compact resolvent. If \(\lambda_n\) and \(v_n\) are its eigenvalues and eigenvectors, then
\[
        u_\alpha(t)=\sum_{n=1}^\infty
        \varphi_n E_{\alpha,1}(-\lambda_n t^\alpha)v_n,
        \qquad \varphi_n=\inner{\varphi}{v_n}.
\]
A scalar observation therefore has the form
\begin{equation}\label{eq:intro-M}
        M_\alpha(t)=\sum_{n=1}^\infty a_nE_{\alpha,1}(-\lambda_n t^\alpha).
\end{equation}
The inverse problem is to determine \(\alpha\) from one or more measured values of \(M_\alpha\).

The main point is that the late-time expansion of \eqref{eq:intro-M} is organized by the resolvent moments
\begin{equation}\label{eq:intro-moments}
        S_k=\sum_{n=1}^\infty \frac{a_n}{\lambda_n^k},\qquad k\ge1.
\end{equation}
If \(S_1\ne0\), the first moment controls the sign of \(\partial_\alpha M_\alpha(t)\) for large \(t\). If \(S_1=0\), uniqueness may still hold: the first non-zero moment \(S_m\) takes over, except at the zeros of the corresponding reciprocal-gamma coefficient. This observation removes an unnecessary limitation in first-moment-only statements.

The contributions of the paper are the following.
\begin{enumerate}[label=(\roman*)]
\item We prove a self-contained uniform differentiated large-time expansion for \(E_{\alpha,1}(-\lambda t^\alpha)\) to arbitrary order, including a remainder estimate after differentiation with respect to the composite fractional parameter.
\item We prove completeness of the resolvent moments: after coefficients belonging to the same eigenspace are grouped, every nontrivial observation has a finite cancellation index \(m=\min\{k:S_k\ne0\}\).
\item We derive eventual strict monotonicity for general signed scalar spectral observations, including the cancellation case \(S_1=\cdots=S_{m-1}=0\), \(S_m\ne0\), and identify the exceptional orders at which the leading reciprocal-gamma coefficient vanishes.
\item For bounded Hilbert-space observations \(M_\alpha(t)=\inner{u_\alpha(t)}{h}\), the moments take the intrinsic form \(S_m=\inner{\cA^{-m}\varphi}{h}\).
\item We prove a two-time log-ratio inverse estimate. If \(\rho>1\), then
\[
        \frac{M_\alpha(\rho t)}{M_\alpha(t)}
        =\rho^{-m\alpha}(1+o(1)),
        \qquad t\to\infty,
\]
and the associated estimator has a combined asymptotic-bias and relative-noise bound. The integer \(m\) is computed from the known forward configuration; it is not an additional unknown inferred from the two data values.
\item For a finite rod, we extend the point-sensor application to arbitrary \(L^2\)-initial data. The leading condition is \((\cA^{-1}\varphi)(x_*)\ne0\), which holds at every interior sensor for nonnegative, nonzero initial temperature.
\item We give explicit finite-dimensional and point-sensor counterexamples showing why eigenspace grouping, a summability hypothesis, reciprocal-gamma exclusions, late observation, stationary-mode subtraction, sensor nondegeneracy, and relative rather than fixed-additive noise control cannot in general be omitted.
\item We provide a reproducible numerical script, explicit quadrature tolerances, and a synthetic point-sensor illustration.
\end{enumerate}

\subsection*{Standard analytic inputs and the resolvent-moment mechanism}

The proof deliberately separates standard analytic background from the
paper-specific order-identification mechanism.  We use the usual spectral
functional calculus for strictly positive self-adjoint operators with compact
resolvent, the standard mild-solution formula for the homogeneous Caputo
evolution, complete monotonicity and rational bounds for
$E_{\alpha,1}(-r)$, and the classical large-argument Mittag--Leffler
asymptotic machinery from the references cited above.  In the rod application
we also use standard one-dimensional elliptic regularity and positivity of
the Dirichlet Green function.  These are background ingredients from the
cited literature; the new part of the paper is how they are combined with a
resolvent-moment cancellation hierarchy for signed observations.

The new argument is the following chain.  First, the differentiated
large-argument expansion is arranged uniformly in the admissible order
interval and in the spectral parameter, so that it may be summed over a
signed observation satisfying the weighted summability condition.  Second,
all coefficients are grouped by eigenspaces, and the moment-completeness
lemma shows that every nontrivial grouped observation has a finite first
nonzero resolvent moment.  Third, this first surviving moment gives eventual
one-time monotonicity away from the relevant reciprocal-gamma zeros.  Fourth,
the same leading term gives the two-time ratio
$M_\alpha(\rho t)/M_\alpha(t)=\rho^{-m\alpha}(1+o(1))$, with a relative-noise
stability estimate.  Finally, in the rod model, the abstract resolvent moment
becomes the point value of $\mathcal A^{-m}\varphi$, so the nondegeneracy
conditions are expressed by elliptic resolvents rather than by individual
Fourier coefficients.

All uniqueness and stability statements are therefore late-time statements for a
known forward configuration: the operator, initial state, observation
functional, and hence the resolvent moments are regarded as known, while the
fractional order is unknown.  This convention is essential in the two-time
ratio formula, where the cancellation index $m$ is computed from the forward
configuration and is not estimated from the two scalar measurements.

The sharpness examples are part of the main structure, not an afterthought:
they show that eigenspace grouping, the summability condition, reciprocal
gamma exclusions, the late-time threshold, stationary-mode subtraction,
sensor nondegeneracy, and the relative-noise interpretation are all genuinely
needed in the stated generality.

\section{Forward problem and scalar observations}\label{sec:forward}

Let \(H\) be a real separable Hilbert space with inner product \(\inner{\cdot}{\cdot}\) and norm \(\norm{\cdot}\).

\begin{assumption}\label{ass:A}
The operator \(\cA:\Dom(\cA)\subset H\to H\) is self-adjoint, strictly positive, and has compact resolvent. Its spectrum consists either of a finite positive list
\[
        0<\lambda_1\le\cdots\le\lambda_N,
\]
with an orthonormal eigenbasis \(\{v_n\}_{n=1}^N\), or, in the infinite-dimensional case, of eigenvalues
\begin{equation}\label{eq:eigenvalues}
        0<\lambda_1\le \lambda_2\le\cdots\le \lambda_n\le\cdots,
        \qquad \lambda_n\to\infty,
\end{equation}
with an orthonormal basis \(\{v_n\}_{n\ge1}\). In either case,
\begin{equation}\label{eq:eigenvectors}
        \cA v_n=\lambda_n v_n.
\end{equation}
In the finite-dimensional case, every spectral sum below is understood to terminate at \(N\).
\end{assumption}

The strict positivity assumption removes stationary modes. If a concrete model has a finite-dimensional kernel, write the observation as \(M_\alpha(t)=M_0+M_\alpha^\perp(t)\), where \(M_0\) is the stationary contribution and \(M_\alpha^\perp\) is obtained by projection onto the orthogonal complement of the kernel. The one-time monotonicity results apply to \(M_\alpha^\perp\), and also to the full observation because \(M_0\) is independent of \(\alpha\). The ratio results, however, must be applied to the projected or stationary-subtracted quantity: if \(M_0\ne0\), then the uncorrected ratio \(M_\alpha(\rho t)/M_\alpha(t)\) tends to \(1\), not to \(\rho^{-m\alpha}\).

For \(0<\alpha<1\) and \(\varphi\in H\), consider
\begin{equation}\label{eq:abstract-problem}
\begin{cases}
        \partial_t^\alpha u_\alpha(t)+\cA u_\alpha(t)=0, & t>0,\\
        u_\alpha(0)=\varphi.
\end{cases}
\end{equation}
Writing \(\varphi_n=\inner{\varphi}{v_n}\), the solution is given by a Mittag-Leffler spectral expansion.

\begin{proposition}\label{prop:forward}
Under Assumption \ref{ass:A}, for every \(\varphi\in H\) and \(\alpha\in(0,1)\), problem \eqref{eq:abstract-problem} has a unique mild solution \(u_\alpha\in C([0,\infty);H)\),
\begin{equation}\label{eq:spectral-solution}
        u_\alpha(t)=\sum_{n=1}^\infty
        \varphi_n E_{\alpha,1}(-\lambda_n t^\alpha)v_n,
        \qquad t\ge0.
\end{equation}
For every finite \(T>0\), the series converges in \(C([0,T];H)\).
\end{proposition}

\begin{proof}
Expanding \(u_\alpha(t)=\sum_n u_n(t)v_n\) reduces \eqref{eq:abstract-problem} to
\[
        \partial_t^\alpha u_n(t)+\lambda_nu_n(t)=0,
        \qquad u_n(0)=\varphi_n.
\]
The scalar Caputo equation has solution \(u_n(t)=\varphi_nE_{\alpha,1}(-\lambda_nt^\alpha)\). Since \(0<E_{\alpha,1}(-r)\le1\) for \(r\ge0\) and \(0<\alpha<1\) by complete monotonicity \cite{Pollard1948,Schneider1996}, Parseval's identity gives convergence in \(H\), uniformly for \(t\in[0,T]\). Uniqueness follows mode by mode.
\end{proof}

We consider scalar observations of the form
\begin{equation}\label{eq:M-def}
        M_\alpha(t)=\sum_{n=1}^\infty a_nE_{\alpha,1}(-\lambda_n t^\alpha),
        \qquad t>0.
\end{equation}
The inverse problem is to recover \(\alpha\) from measured data such as
\begin{equation}\label{eq:one-data}
        M_\alpha(t_0)=d_0
\end{equation}
or from a two-time pair \(M_\alpha(t_0)=d_1\), \(M_\alpha(\rho t_0)=d_2\) with \(\rho>1\).

The following coefficient condition is used for general spectral observations.

\begin{assumption}\label{ass:moment-finite}
The coefficients \(\{a_n\}_{n\ge1}\subset\R\) satisfy
\begin{equation}\label{eq:first-absolute-moment}
        \sum_{n=1}^\infty \frac{\abs{a_n}}{\lambda_n}<\infty.
\end{equation}
For \(k\ge1\), define the resolvent moments
\begin{equation}\label{eq:S-k-def}
        S_k=\sum_{n=1}^\infty \frac{a_n}{\lambda_n^k}.
\end{equation}
\end{assumption}

Condition \eqref{eq:first-absolute-moment} implies that every \(S_k\) is finite, because \(\lambda_n\ge\lambda_1>0\). The same assumption also makes the observation \(M_\alpha(t)\) well defined for positive times. Indeed, the standard estimate
\begin{equation}\label{eq:ML-rational-bound}
        \abs{E_{\alpha,1}(-r)}\le \frac{C_I}{1+r},
        \qquad r\ge0,\quad \alpha\in I\Subset(0,1),
\end{equation}
which follows from the complete monotonicity and integral representations of the Mittag-Leffler function \cite{Pollard1948,GorenfloKilbasMainardiRogosin2020}, implies absolute and locally uniform convergence of \eqref{eq:M-def} on \(I\times[\tau,\infty)\) for every \(\tau>0\). In particular, the late-time asymptotic manipulations below are applied to a genuine scalar observation, not only to a formal series.

\subsection{Completeness of the resolvent moments}

Because eigenvalues may have multiplicity, coefficients in the same eigenspace must first be grouped. Let \(\{\mu_j\}_{j\in J}\) be the distinct eigenvalues of \(\cA\), and set
\begin{equation}\label{eq:grouped-coefficients}
        B_j=\sum_{\{n:\lambda_n=\mu_j\}}a_n.
\end{equation}
Each sum is finite, and Assumption \ref{ass:moment-finite} gives
\[
        \sum_{j\in J}\frac{\abs{B_j}}{\mu_j}
        \le \sum_{n=1}^\infty\frac{\abs{a_n}}{\lambda_n}<\infty.
\]
Hence the observation and its moments can be regrouped as
\begin{equation}\label{eq:grouped-observation-moments}
        M_\alpha(t)=\sum_{j\in J}B_jE_{\alpha,1}(-\mu_jt^\alpha),
        \qquad
        S_k=\sum_{j\in J}\frac{B_j}{\mu_j^k}.
\end{equation}

\begin{lemma}[Moment completeness and the cancellation index]\label{lem:moment-completeness}
Under Assumption \ref{ass:moment-finite}, if at least one grouped coefficient \(B_j\) is nonzero, then the moments \(S_k\) cannot all vanish. Consequently,
\begin{equation}\label{eq:cancellation-index}
        m_*:=\min\{k\in\N:S_k\ne0\}
\end{equation}
is a finite positive integer for every nontrivial grouped spectral observation.
\end{lemma}

\begin{proof}
Define a finite signed Borel measure on \([0,\lambda_1^{-1}]\) by
\begin{equation}\label{eq:moment-measure}
        \nu=\sum_{j\in J}\frac{B_j}{\mu_j}\,\delta_{1/\mu_j}.
\end{equation}
Its total variation is finite by the preceding estimate. If \(S_k=0\) for every \(k\ge1\), then
\[
        \int_0^{\lambda_1^{-1}}x^{k-1}\dd\nu(x)=S_k=0,
        \qquad k\ge1.
\]
Thus \(\nu\) annihilates every polynomial. By the Weierstrass approximation theorem, polynomials are dense in \(C([0,\lambda_1^{-1}])\); since integration against \(\nu\) is a bounded linear functional, \(\nu\) annihilates every continuous function. Hence \(\nu=0\). Its atoms are at the distinct points \(1/\mu_j\), so \(B_j/\mu_j=0\) for every \(j\), contradicting the assumed nontriviality. Therefore at least one moment is nonzero, and the well-ordering of \(\N\) gives \eqref{eq:cancellation-index}.
\end{proof}

\begin{remark}[The index \(m\) is part of the known forward configuration]\label{rem:m-known}
The inverse problem considered here treats \(\cA\), the initial state, and the observation functional as known; only \(\alpha\) is unknown. Accordingly, the coefficients \(a_n\), the moments \(S_k\), and the cancellation index \(m=m_*\) are determined before the two-time estimator is applied. The factor \(m\) in that estimator is therefore not inferred from two scalar measurements. If the forward configuration is itself uncertain, joint estimation of \(m\) or of the coefficients becomes a different inverse problem.
\end{remark}

\subsection{Bounded Hilbert-space observations}

A particularly natural case is
\begin{equation}\label{eq:bounded-observation}
        M_\alpha(t)=\inner{u_\alpha(t)}{h},
        \qquad h\in H.
\end{equation}
Then \(a_n=\varphi_nh_n\), where \(h_n=\inner{h}{v_n}\). By Cauchy's inequality,
\begin{equation}\label{eq:bounded-summability}
        \sum_{n=1}^\infty \frac{\abs{a_n}}{\lambda_n}
        \le \frac{1}{\lambda_1}
        \left(\sum_{n=1}^\infty \abs{\varphi_n}^2\right)^{1/2}
        \left(\sum_{n=1}^\infty \abs{h_n}^2\right)^{1/2}
        =\frac{\norm{\varphi}\norm{h}}{\lambda_1}.
\end{equation}
Moreover,
\begin{equation}\label{eq:resolvent-moment}
        S_k=\sum_{n=1}^\infty\frac{\varphi_nh_n}{\lambda_n^k}
        =\inner{\cA^{-k}\varphi}{h},
        \qquad k\ge1.
\end{equation}
Thus the non-degeneracy conditions below can be written intrinsically in terms of the resolvents of \(\cA\).

\section{Differentiated Mittag-Leffler asymptotics}\label{sec:ml}

The parameter differentiability of Mittag-Leffler-type functions is discussed in \cite{RogosinGiraldiMainardi2025}. The result below is tailored to the inverse problem: it controls the large-argument expansion and its \(\alpha\)-derivative uniformly in both the order and the spectral parameter.

Set
\begin{equation}\label{eq:c-k}
        c_k(\alpha)=\frac{(-1)^{k+1}}{\Gamma(1-k\alpha)},
        \qquad k\in\N.
\end{equation}
Since \(1/\Gamma\) is an entire function, \(c_k\) is smooth on \((0,1)\), even when \(1-k\alpha\) is a non-positive integer.

\begin{theorem}[Uniform differentiated expansion]\label{thm:ML-expansion}
Let \(I=[\alpha_0,\alpha_1]\Subset(0,1)\), \(\lambda_0>0\), and \(N\in\N\). There exist constants \(C=C(I,\lambda_0,N)>0\) and \(T=T(I,\lambda_0,N)>1\) such that, for all \(\alpha\in I\), \(\lambda\ge\lambda_0\), and \(t\ge T\),
\begin{equation}\label{eq:ML-expansion-N}
        E_{\alpha,1}(-\lambda t^\alpha)
        =\sum_{k=1}^{N}c_k(\alpha)\lambda^{-k}t^{-k\alpha}+R_N(\alpha,\lambda,t),
\end{equation}
where
\begin{equation}\label{eq:R-N-bound}
        \abs{R_N(\alpha,\lambda,t)}
        \le C\lambda^{-(N+1)}t^{-(N+1)\alpha}.
\end{equation}
Furthermore,
\begin{equation}\label{eq:dML-expansion-N}
\begin{aligned}
        \partial_\alpha E_{\alpha,1}(-\lambda t^\alpha)
        &=\sum_{k=1}^{N}\lambda^{-k}t^{-k\alpha}
          \bigl(c_k'(\alpha)-k c_k(\alpha)\log t\bigr)\\
        &\quad +\partial_\alpha R_N(\alpha,\lambda,t),
\end{aligned}
\end{equation}
with
\begin{equation}\label{eq:dR-N-bound}
        \abs{\partial_\alpha R_N(\alpha,\lambda,t)}
        \le C(1+\log t)\lambda^{-(N+1)}t^{-(N+1)\alpha}.
\end{equation}
\end{theorem}

\begin{proof}
For clarity, write
\[
        F(\alpha,x)=E_{\alpha,1}(-x),\qquad x>0.
\]
Pollard's completely monotone representation, followed by the change of variables
\(s=xr^\alpha\), gives
\begin{equation}\label{eq:pollard-smoothed}
        F(\alpha,x)
        =\frac{\sin(\pi\alpha)}{\pi\alpha x}
        \int_0^\infty
        \frac{e^{-s^{1/\alpha}}}
        {1+2(s/x)\cos(\pi\alpha)+(s/x)^2}\dd s;
\end{equation}
see \cite{Pollard1948,GorenfloKilbasMainardiRogosin2020}.

Put \(\theta=\pi\alpha\), \(y=s/x\), and
\[
        D_\alpha(y)=1+2y\cos\theta+y^2.
\]
For \(j\ge0\), define
\[
        b_j(\alpha)=(-1)^j\frac{\sin((j+1)\theta)}{\sin\theta}.
\]
These coefficients satisfy \(b_0=1\), \(b_1=-2\cos\theta\), and
\(b_j+2\cos\theta\,b_{j-1}+b_{j-2}=0\). Hence the following finite identity is exact for every \(y\ge0\):
\begin{equation}\label{eq:rational-finite-expansion}
        \frac{1}{D_\alpha(y)}
        =\sum_{j=0}^{N-1}b_j(\alpha)y^j
        +y^N H_N(\alpha,y),
\end{equation}
where
\begin{equation}\label{eq:H-N-def}
        H_N(\alpha,y)
        =\frac{b_N(\alpha)-b_{N-1}(\alpha)y}{D_\alpha(y)}.
\end{equation}
Because \(I\Subset(0,1)\),
\[
        D_\alpha(y)
        \ge d_I(1+y^2),
        \qquad
        d_I=1-\max_{\alpha\in I}\abs{\cos(\pi\alpha)}>0.
\]
It follows directly from \eqref{eq:H-N-def} that
\begin{equation}\label{eq:H-N-bounds}
        \sup_{\alpha\in I,\,y\ge0}
        \left(
        \abs{H_N(\alpha,y)}
        +\abs{\partial_\alpha H_N(\alpha,y)}
        +\abs{y\partial_yH_N(\alpha,y)}
        \right)<\infty.
\end{equation}

For \(j\ge0\), the substitution \(r=s^{1/\alpha}\) yields
\begin{equation}\label{eq:exponential-moment}
        \int_0^\infty e^{-s^{1/\alpha}}s^j\dd s
        =\alpha\Gamma((j+1)\alpha).
\end{equation}
Inserting \eqref{eq:rational-finite-expansion} into
\eqref{eq:pollard-smoothed}, using \eqref{eq:exponential-moment}, and setting
\(k=j+1\), the coefficient of \(x^{-k}\) is
\[
        \frac{\sin(\pi\alpha)}{\pi}
        b_{k-1}(\alpha)\Gamma(k\alpha)
        =\frac{(-1)^{k+1}}{\Gamma(1-k\alpha)}
        =c_k(\alpha),
\]
where the reflection formula is used, with continuity at the zeros of
\(1/\Gamma(1-k\alpha)\). Thus
\begin{equation}\label{eq:F-expansion-with-rem}
        F(\alpha,x)=\sum_{k=1}^N c_k(\alpha)x^{-k}
        +\mathcal R_N(\alpha,x),
\end{equation}
with the exact remainder
\begin{equation}\label{eq:calR-N}
        \mathcal R_N(\alpha,x)
        =\frac{\sin(\pi\alpha)}{\pi\alpha}
        x^{-N-1}
        \int_0^\infty e^{-s^{1/\alpha}}s^N
        H_N(\alpha,s/x)\dd s.
\end{equation}

The bounds \eqref{eq:H-N-bounds} and dominated differentiation imply
\begin{equation}\label{eq:calR-derivative-bounds}
        \abs{\mathcal R_N(\alpha,x)}
        +\abs{\partial_\alpha\mathcal R_N(\alpha,x)}
        +\abs{x\partial_x\mathcal R_N(\alpha,x)}
        \le C_{I,N}x^{-N-1},
        \qquad \alpha\in I,
\end{equation}
for every \(x>0\). Indeed,
\[
        \partial_\alpha e^{-s^{1/\alpha}}
        =\alpha^{-2}s^{1/\alpha}(\log s)e^{-s^{1/\alpha}}.
\]
For \(0<s\le1\), the resulting integrands are bounded by a constant multiple of
\(s^N(1+s^{1/\alpha_1}\abs{\log s})\). For \(s\ge1\), they are bounded by a constant multiple of
\[
        s^N\bigl(1+s^{1/\alpha_0}\log s\bigr)e^{-s^{1/\alpha_1}}.
\]
Both majorants are integrable. This proves uniform dominated differentiation in \(\alpha\); the \(x\)-derivative is controlled by the last term in
\eqref{eq:H-N-bounds}.

Finally set \(x=\lambda t^\alpha\) and
\(R_N(\alpha,\lambda,t)=\mathcal R_N(\alpha,\lambda t^\alpha)\).
Then \eqref{eq:R-N-bound} follows from \eqref{eq:calR-derivative-bounds}, while the chain rule gives
\[
        \partial_\alpha R_N(\alpha,\lambda,t)
        =\partial_\alpha\mathcal R_N(\alpha,x)
        +x\log t\,\partial_x\mathcal R_N(\alpha,x).
\]
For \(t>1\), this proves \eqref{eq:dR-N-bound}. Differentiating the finite sum in
\eqref{eq:F-expansion-with-rem} gives \eqref{eq:dML-expansion-N} and completes the proof.
\end{proof}

For \(N=1\), Theorem \ref{thm:ML-expansion} gives the explicit leading derivative
\begin{equation}\label{eq:first-derivative-leading}
        \partial_\alpha E_{\alpha,1}(-\lambda t^\alpha)
        =\frac{t^{-\alpha}}{\lambda\Gamma(1-\alpha)}
          \bigl(\psi(1-\alpha)-\log t\bigr)
          +O\left((1+\log t)\lambda^{-2}t^{-2\alpha}\right),
\end{equation}
where \(\psi=\Gamma'/\Gamma\). Since \(1/\Gamma(1-\alpha)>0\) on \((0,1)\), this immediately implies one-mode eventual monotonicity.

\begin{corollary}[One-mode monotonicity]\label{cor:one-mode}
Let \(I\Subset(0,1)\) and \(\lambda_0>0\). There exists \(T_*=T_*(I,\lambda_0)>1\) such that, for every \(\lambda\ge\lambda_0\) and \(t\ge T_*\),
\begin{equation}\label{eq:one-mode-negative}
        \partial_\alpha E_{\alpha,1}(-\lambda t^\alpha)<0,
        \qquad \alpha\in I.
\end{equation}
Thus \(\alpha\mapsto E_{\alpha,1}(-\lambda t^\alpha)\) is strictly decreasing on \(I\).
\end{corollary}

\section{One-measurement identification and cancellation}\label{sec:one-measurement}

We now apply the expansion to \(M_\alpha\). Under Assumption \ref{ass:moment-finite}, Theorem \ref{thm:ML-expansion} can be summed term by term.

\begin{theorem}[Observation asymptotics]\label{thm:obs-expansion}
Let Assumptions \ref{ass:A} and \ref{ass:moment-finite} hold. Let \(I=[\alpha_0,\alpha_1]\Subset(0,1)\) and \(N\in\N\). Then, uniformly for \(\alpha\in I\),
\begin{equation}\label{eq:M-expansion-N}
        M_\alpha(t)=\sum_{k=1}^{N}c_k(\alpha)S_kt^{-k\alpha}
        +O(t^{-(N+1)\alpha}),
        \qquad t\to\infty,
\end{equation}
and
\begin{equation}\label{eq:dM-expansion-N}
        \partial_\alpha M_\alpha(t)
        =\sum_{k=1}^{N}S_kt^{-k\alpha}
        \bigl(c_k'(\alpha)-kc_k(\alpha)\log t\bigr)
        +O((1+\log t)t^{-(N+1)\alpha}).
\end{equation}
\end{theorem}

\begin{proof}
Apply Theorem \ref{thm:ML-expansion} with \(\lambda_0=\lambda_1\). Because \eqref{eq:first-absolute-moment} holds and \(\lambda_n\ge\lambda_1\), the sums
\[
        \sum_{n=1}^\infty\frac{\abs{a_n}}{\lambda_n^k},
        \qquad k\ge1,
\]
are finite. Hence the finite asymptotic terms and the remainders are summable. Summing the leading terms gives the moments \(S_k\), and the uniform remainder bounds follow from
\[
        \sum_{n=1}^\infty\frac{\abs{a_n}}{\lambda_n^{N+1}}
        \le \lambda_1^{-N}
        \sum_{n=1}^\infty\frac{\abs{a_n}}{\lambda_n}.
\]
Moreover, \eqref{eq:dML-expansion-N} and \eqref{eq:dR-N-bound} imply absolute uniform convergence of the differentiated series for \(\alpha\in I\) and \(t\ge T\), after multiplication by \(a_n\) and summation. Termwise differentiation is therefore justified by the Weierstrass theorem, and \eqref{eq:dM-expansion-N} follows.
\end{proof}

By Lemma \ref{lem:moment-completeness}, every nontrivial grouped observation has a finite cancellation index. The next theorem describes eventual monotonicity once this first nonzero moment is identified; it includes the usual first-moment case and all higher cancellation cases.

\begin{theorem}[Eventual monotonicity from the first non-zero moment]\label{thm:first-nonzero}
Let Assumptions \ref{ass:A} and \ref{ass:moment-finite} hold, and fix \(I=[\alpha_0,\alpha_1]\Subset(0,1)\). Suppose that for some \(m\in\N\),
\begin{equation}\label{eq:first-nonzero-moment}
        S_1=\cdots=S_{m-1}=0,
        \qquad S_m\ne0.
\end{equation}
Assume also that
\begin{equation}\label{eq:no-gamma-zero}
        c_m(\alpha)=\frac{(-1)^{m+1}}{\Gamma(1-m\alpha)}\ne0,
        \qquad \alpha\in I.
\end{equation}
Equivalently, if \(m\ge2\), the interval \(I\) does not intersect
\begin{equation}\label{eq:excluded-points}
        \left\{\frac{1}{m},\frac{2}{m},\ldots,\frac{m-1}{m}\right\}.
\end{equation}
Then there exists \(T_I>1\) such that, for every \(t_0\ge T_I\), the map
\[
        I\ni\alpha\longmapsto M_\alpha(t_0)
\]
is strictly monotone. More precisely, its monotonicity direction is determined by the constant sign of \(-S_mc_m(\alpha)\) on \(I\). Consequently the equation
\begin{equation}\label{eq:one-measurement-equation}
        M_\alpha(t_0)=d_0
\end{equation}
has at most one solution \(\alpha\in I\).
\end{theorem}

\begin{proof}
Take \(N=m\) in Theorem \ref{thm:obs-expansion}. Since the first \(m-1\) moments vanish,
\begin{equation}\label{eq:M-first-nonzero}
        M_\alpha(t)=S_mc_m(\alpha)t^{-m\alpha}+O(t^{-(m+1)\alpha})
\end{equation}
and
\begin{equation}\label{eq:dM-first-nonzero}
        \partial_\alpha M_\alpha(t)
        =S_mt^{-m\alpha}\bigl(c_m'(\alpha)-mc_m(\alpha)\log t\bigr)
        +O((1+\log t)t^{-(m+1)\alpha}).
\end{equation}
By \eqref{eq:no-gamma-zero}, the continuous function \(c_m\) has a constant non-zero sign on the interval \(I\), and
\[
        c_*:=\min_{\alpha\in I}\abs{c_m(\alpha)}>0.
\]
The term \(-mS_mc_m(\alpha)t^{-m\alpha}\log t\) dominates \eqref{eq:dM-first-nonzero} uniformly as \(t\to\infty\). Therefore \(\partial_\alpha M_\alpha(t)\) has the constant sign \(-\operatorname{sign}(S_mc_m)\) on \(I\) for every sufficiently large \(t\). Strict monotonicity and uniqueness follow.
\end{proof}

\begin{corollary}[The first resolvent moment]\label{cor:S1}
If \(S_1\ne0\), then the conclusion of Theorem \ref{thm:first-nonzero} holds with \(m=1\), and no exceptional values of \(\alpha\) have to be excluded. The map \(\alpha\mapsto M_\alpha(t_0)\) is strictly decreasing for large \(t_0\) if \(S_1>0\) and strictly increasing if \(S_1<0\).
\end{corollary}

\begin{corollary}[Bounded observations in intrinsic form]\label{cor:bounded-intrinsic}
Let \(M_\alpha(t)=\inner{u_\alpha(t)}{h}\), where \(\varphi,h\in H\). Suppose that for some \(m\in\N\),
\begin{equation}\label{eq:intrinsic-first-nonzero}
        \inner{\cA^{-k}\varphi}{h}=0
        \quad (1\le k\le m-1),
        \qquad
        \inner{\cA^{-m}\varphi}{h}\ne0,
\end{equation}
where the vanishing condition is void when \(m=1\). Assume also that \(I\) satisfies \eqref{eq:no-gamma-zero}. Then one sufficiently late value of \(\inner{u_\alpha(t)}{h}\) determines at most one \(\alpha\in I\).
\end{corollary}

\begin{proof}
By \eqref{eq:bounded-summability} the general summability condition holds, and by \eqref{eq:resolvent-moment} the spectral moments are the resolvent moments. The result follows from Theorem \ref{thm:first-nonzero}.
\end{proof}

\begin{remark}[Near cancellation and pre-asymptotic crossover]\label{rem:near-cancellation}
The cancellation index is an exact structural quantity and may change under small perturbations of the forward configuration. For example, suppose at a fixed admissible order that \(S_1=\varepsilon\ne0\) is very small, \(S_2=\cdots=S_{m-1}=0\), and \(S_m\ne0\) for some \(m\ge2\). Away from zeros of \(c_1\) and \(c_m\), the first and \(m\)-th asymptotic terms become comparable on the scale
\[
        t^{(m-1)\alpha}
        \asymp
        \frac{\abs{c_m(\alpha)S_m}}
             {\abs{c_1(\alpha)\varepsilon}}.
\]
Before this crossover, finite-time data may resemble an \(m\)-th order cancellation even though the true eventual index is \(1\). Exact cancellations should therefore be verified from the known model, and near-cancellation should be accounted for when choosing observation times or interpreting an empirical decay slope.
\end{remark}

The proof also gives a quantitative sufficient condition for the time threshold. We state it in a form that separates the asymptotic constant from the data.

\begin{proposition}[A sufficient late-time condition]\label{prop:quantitative-time}
Let the assumptions of Theorem \ref{thm:first-nonzero} hold. Let \(C_I\) be a constant such that the remainder in \eqref{eq:dM-first-nonzero} is bounded by
\[
        C_I(1+\log t)t^{-(m+1)\alpha}
        \sum_{n=1}^\infty\frac{\abs{a_n}}{\lambda_n^{m+1}}
\]
for \(\alpha\in I\) and \(t\ge T_0\). Put
\[
        A_{m+1}=\sum_{n=1}^\infty\frac{\abs{a_n}}{\lambda_n^{m+1}},
        \quad
        c_* =\min_{\alpha\in I}\abs{c_m(\alpha)},
        \quad
        c^* =\sup_{\alpha\in I}\abs{c_m'(\alpha)}.
\]
Then monotonicity on \(I\) holds for every \(t\ge T_0\) satisfying
\begin{equation}\label{eq:quantitative-time}
        \abs{S_m}\bigl(mc_*\log t-c^*\bigr)
        > C_IA_{m+1}(1+\log t)t^{-\alpha_0}.
\end{equation}
\end{proposition}

\begin{proof}
For fixed \(\alpha\in I\), the leading part of \eqref{eq:dM-first-nonzero} has magnitude at least
\[
        \abs{S_m}t^{-m\alpha}(mc_*\log t-c^*)
\]
whenever the factor in parentheses is positive, while the derivative remainder is bounded by
\[
        C_IA_{m+1}(1+\log t)t^{-(m+1)\alpha}.
\]
After division by \(t^{-m\alpha}\), it is enough to require
\[
        \abs{S_m}(mc_*\log t-c^*)
        > C_IA_{m+1}(1+\log t)t^{-\alpha}.
\]
Since \(t^{-\alpha}\le t^{-\alpha_0}\) for \(\alpha\in I\), condition \eqref{eq:quantitative-time} is sufficient. It also holds for all sufficiently large \(t\), because the left-hand side grows like \(\log t\) and the right-hand side tends to zero. Thus the derivative cannot change sign on \(I\).
\end{proof}

The same derivative lower bound gives Lipschitz stability for the one-measurement inverse map.

\begin{corollary}[One-measurement Lipschitz stability]\label{cor:one-measurement-stability}
Under the assumptions of Theorem \ref{thm:first-nonzero}, let \(t_0\ge T_I\) and define
\begin{equation}\label{eq:m-I-t}
        \mu_I(t_0)=\inf_{\alpha\in I}\abs{\partial_\alpha M_\alpha(t_0)}.
\end{equation}
Then \(\mu_I(t_0)>0\), and for all \(\alpha,\widetilde\alpha\in I\),
\begin{equation}\label{eq:one-measurement-lip}
        \abs{\alpha-\widetilde\alpha}
        \le \frac{1}{\mu_I(t_0)}
        \abs{M_\alpha(t_0)-M_{\widetilde\alpha}(t_0)}.
\end{equation}
Moreover, after increasing \(T_I\) if necessary,
\begin{equation}\label{eq:mu-lower-general}
        \mu_I(t_0)
        \ge \frac{m\abs{S_m}}{2}
        \left(\min_{\alpha\in I}\abs{c_m(\alpha)}\right)
        t_0^{-m\alpha_1}\log t_0.
\end{equation}
\end{corollary}

\begin{proof}
The mean-value theorem gives \eqref{eq:one-measurement-lip}. Estimate \eqref{eq:mu-lower-general} follows from \eqref{eq:dM-first-nonzero} because the logarithmic leading term dominates the bounded derivative of \(c_m\) and the remainder for large \(t_0\).
\end{proof}

\begin{remark}[Late-time tradeoff]\label{rem:tradeoff}
The sensitivity quantity \(\mu_I(t_0)\) itself tends to zero as \(t_0\to\infty\). Indeed, for any fixed \(\bar\alpha\in I\),
\[
        0\le \mu_I(t_0)
        \le \abs{\partial_\alpha M_{\bar\alpha}(t_0)}
        =O\bigl((1+\log t_0)t_0^{-m\bar\alpha}\bigr)
        \longrightarrow0.
\]
Thus late observation enforces monotonicity, but excessive delay reduces sensitivity to absolute perturbations in a single measurement. The two-time ratio below removes the leading amplitude decay in the logarithmic ratio variable; its time-uniform interpretation is therefore relative rather than additive.
\end{remark}

\section{Two-time log-ratio identification}\label{sec:ratio}

A single late value is sufficient for uniqueness but can be poorly conditioned in the absolute data variable. Measuring at two late times and taking a logarithmic ratio cancels the leading amplitude and produces a better-conditioned relative data variable.

Fix \(\rho>1\) and define
\begin{equation}\label{eq:R-def}
        R_\alpha(t,\rho)=\frac{M_\alpha(\rho t)}{M_\alpha(t)},
\end{equation}
whenever the denominator is non-zero.

\begin{theorem}[Two-time log-ratio inverse estimate]\label{thm:ratio}
Assume the hypotheses of Theorem \ref{thm:first-nonzero}. Let \(\rho>1\). There exists \(T_{I,\rho}>1\) such that, for all \(t\ge T_{I,\rho}\), \(M_\alpha(t)\) and \(M_\alpha(\rho t)\) are non-zero on \(I\), and
\begin{equation}\label{eq:ratio-asymp}
        R_\alpha(t,\rho)=\rho^{-m\alpha}\bigl(1+O(t^{-\alpha_0})\bigr)
\end{equation}
uniformly for \(\alpha\in I\). Moreover,
\begin{equation}\label{eq:logratio-derivative}
        \partial_\alpha\log\abs{R_\alpha(t,\rho)}
        =-m\log\rho+O((1+\log t)t^{-\alpha_0})
\end{equation}
uniformly for \(\alpha\in I\). Consequently, for every sufficiently large \(t\), the map
\[
        \alpha\longmapsto \log\abs{R_\alpha(t,\rho)}
\]
is strictly decreasing on \(I\), and
\begin{equation}\label{eq:ratio-stability}
        \abs{\alpha-\widetilde\alpha}
        \le \frac{2}{m\log\rho}
        \left\lvert
        \log\abs{R_\alpha(t,\rho)}
        -\log\abs{R_{\widetilde\alpha}(t,\rho)}
        \right\rvert
\end{equation}
for all \(\alpha,\widetilde\alpha\in I\).
\end{theorem}

\begin{proof}
By \eqref{eq:M-first-nonzero},
\begin{equation}\label{eq:M-relative}
        M_\alpha(t)=S_mc_m(\alpha)t^{-m\alpha}\bigl(1+r_\alpha(t)\bigr),
        \qquad r_\alpha(t)=O(t^{-\alpha_0}),
\end{equation}
uniformly for \(\alpha\in I\). The non-vanishing of \(S_mc_m(\alpha)\) and the uniform smallness of \(r_\alpha(t)\) imply that \(M_\alpha(t)\) and \(M_\alpha(\rho t)\) are non-zero for large \(t\). Dividing the expansion at \(\rho t\) by the expansion at \(t\) gives \eqref{eq:ratio-asymp}. Differentiating the logarithm of \eqref{eq:M-relative}, and using \eqref{eq:dM-first-nonzero}, gives
\[
        \partial_\alpha\log\abs{M_\alpha(t)}
        =\frac{c_m'(\alpha)}{c_m(\alpha)}-m\log t
        +O((1+\log t)t^{-\alpha_0}).
\]
Subtracting the same formula with \(t\) replaced by \(\rho t\) proves \eqref{eq:logratio-derivative}. For large \(t\), the error in \eqref{eq:logratio-derivative} is bounded by \((m\log\rho)/2\). The mean-value theorem then yields \eqref{eq:ratio-stability}.
\end{proof}

The leading reconstruction formula is therefore
\begin{equation}\label{eq:log-ratio-estimator}
        \alpha_{\rm LR}
        =-\frac{1}{m\log\rho}\log\left\lvert\frac{d_2}{d_1}\right\rvert,
        \qquad d_1=M_\alpha(t),\quad d_2=M_\alpha(\rho t),
\end{equation}
with a model error of order \(O(t^{-\alpha_0})\) under the assumptions of Theorem \ref{thm:ratio}. The integer \(m\) is the known cancellation index from \eqref{eq:cancellation-index}; two measurements alone do not determine it. When \(S_1\ne0\), one has \(m=1\), and \eqref{eq:log-ratio-estimator} is the secant, in logarithmic time, of the late-time logarithmic-slope reconstruction used in \cite{HatanoNakagawaWangYamamoto2013}.

\begin{proposition}[Relative-noise perturbation]\label{prop:relative-noise}
Let \(d_1,d_2\ne0\), and suppose that the measured values are
\[
        \widetilde d_j=d_j(1+\varepsilon_j),
        \qquad \abs{\varepsilon_j}\le\eta<1,
        \qquad j=1,2.
\]
Let \(\alpha_{\rm LR}\) and \(\widetilde\alpha_{\rm LR}\) be the estimators obtained from \eqref{eq:log-ratio-estimator} using \((d_1,d_2)\) and \((\widetilde d_1,\widetilde d_2)\), respectively. Then
\begin{equation}\label{eq:relative-noise-bound}
        \abs{\widetilde\alpha_{\rm LR}-\alpha_{\rm LR}}
        \le \frac{2[-\log(1-\eta)]}{m\log\rho}.
\end{equation}
In particular, the perturbation is \(O(\eta)\) as \(\eta\to0\), with a constant independent of the late observation time.
\end{proposition}

\begin{proof}
Since \(1+\varepsilon_j>0\),
\[
\begin{aligned}
        \abs{\widetilde\alpha_{\rm LR}-\alpha_{\rm LR}}
        &\le \frac{1}{m\log\rho}
        \left(
        \abs{\log(1+\varepsilon_1)}
        +\abs{\log(1+\varepsilon_2)}
        \right)\\
        &\le \frac{2[-\log(1-\eta)]}{m\log\rho}.
\end{aligned}
\]
\end{proof}

\begin{corollary}[Combined model and relative-noise error]\label{cor:combined-error}
Under the hypotheses of Theorem \ref{thm:ratio}, there exist constants \(C>0\) and \(T>1\), depending only on the admissible interval, \(\rho\), and the known forward configuration, such that the following holds. Let \(\alpha\in I\),
\[
        d_1=M_\alpha(t),\qquad d_2=M_\alpha(\rho t),
\]
and let \(\widetilde d_j=d_j(1+\varepsilon_j)\) with \(\abs{\varepsilon_j}\le\eta<1\). For every \(t\ge T\), the estimator formed from the noisy values satisfies
\begin{equation}\label{eq:combined-error}
        \abs{\widetilde\alpha_{\rm LR}-\alpha}
        \le C t^{-\alpha_0}
        +\frac{2[-\log(1-\eta)]}{m\log\rho}.
\end{equation}
\end{corollary}

\begin{proof}
Equation \eqref{eq:ratio-asymp} and the uniform smallness of its remainder give
\[
        \abs{\alpha_{\rm LR}-\alpha}\le Ct^{-\alpha_0}
\]
for all sufficiently large \(t\), uniformly for \(\alpha\in I\). Combine this estimate with Proposition \ref{prop:relative-noise} and the triangle inequality.
\end{proof}

\begin{remark}[Fixed additive noise]\label{rem:additive-noise}
The preceding estimate is not a time-uniform bound for fixed additive sensor errors. If \(\widetilde d_j=d_j+\delta_j\), then the corresponding relative errors are \(\varepsilon_j=\delta_j/d_j\). Under the hypotheses of Theorem \ref{thm:ratio}, \(\abs{d_j}\asymp t^{-m\alpha}\), so a fixed absolute error may produce \(\abs{\varepsilon_j}\asymp \abs{\delta_j}t^{m\alpha}\). Consequently, the log-ratio method should be used while both measurements remain safely above the additive noise floor.
\end{remark}

\section{Sharpness and counterexamples for the abstract theory}\label{sec:sharpness-abstract}

The preceding results are asymptotic and conditional. This section tests the hypotheses by examples. Except where a hypothesis is explicitly removed, the finite-dimensional constructions satisfy Assumptions \ref{ass:A} and \ref{ass:moment-finite}; hence they are genuine models within the abstract framework. A counterexample below is directed at a strengthened statement, not at the theorem as stated.

\subsection{Summability and grouping}

\begin{example}[Why equal eigenvalues must be grouped]\label{ex:sharp-grouping}
Let \(H=\R^2\), \(\cA=I\), \(\varphi=(1,1)\), and \(h=(1,-1)\). Then the two ungrouped coefficients are nonzero, but
\[
        M_\alpha(t)=E_{\alpha,1}(-t^\alpha)-E_{\alpha,1}(-t^\alpha)=0,
        \qquad
        S_k=1-1=0
\]
for every \(k\ge1\). The only distinct eigenvalue is \(1\), and its grouped coefficient is \(B=1-1=0\). Thus Lemma \ref{lem:moment-completeness} would be false if nonzero individual coefficients were confused with a nontrivial grouped observation.
\end{example}

\begin{example}[Failure without weighted summability]\label{ex:sharp-summability}
Let \(H=\ell^2\), \(\cA e_n=n e_n\), and consider the formal coefficients \(a_n=1\). Then \(\cA\) is strictly positive, self-adjoint, and has compact resolvent, but
\[
        \sum_{n=1}^\infty \frac{|a_n|}{\lambda_n}
        =\sum_{n=1}^\infty\frac1n=\infty.
\]
For every fixed \(t>0\) and \(\alpha\in(0,1)\), the first large-argument term gives
\[
        nE_{\alpha,1}(-nt^\alpha)
        \longrightarrow \frac{t^{-\alpha}}{\Gamma(1-\alpha)}>0.
\]
Consequently \(\sum_nE_{\alpha,1}(-nt^\alpha)\) diverges by comparison with the harmonic series. Some summability condition is therefore necessary for a general scalar spectral observation to be well defined.
\end{example}

\subsection{The reciprocal-gamma exceptional set}

\begin{proposition}[Sharpness of the exceptional order for \(m=2\)]\label{prop:sharp-gamma}
Let
\[
        H=\R^2,
        \qquad
        \cA=\mathrm{diag}(1,2),
        \qquad
        \varphi=(1,1),
        \qquad
        h=(1,-2).
\]
Then
\begin{equation}\label{eq:sharp-two-mode-M}
        M_\alpha(t)
        =E_{\alpha,1}(-t^\alpha)
        -2E_{\alpha,1}(-2t^\alpha)
\end{equation}
and the cancellation index is \(m=2\). For every
\(0<\alpha_-<1/2<\alpha_+<1\), there exists \(T>1\) such that, for all \(t\ge T\),
\begin{equation}\label{eq:sharp-opposite-derivatives}
        \partial_\alpha M_{\alpha_-}(t)>0,
        \qquad
        \partial_\alpha M_{\alpha_+}(t)<0.
\end{equation}
Hence \(\alpha\mapsto M_\alpha(t)\) is not injective on \([\alpha_-,\alpha_+]\) for such \(t\). At the exceptional order \(\alpha=1/2\),
\begin{equation}\label{eq:sharp-half-asymptotic}
        M_{1/2}(t)
        =-\frac{3}{8\sqrt\pi}\,t^{-3/2}+O(t^{-2}),
\end{equation}
and therefore
\begin{equation}\label{eq:sharp-half-ratio}
        \frac{M_{1/2}(\rho t)}{M_{1/2}(t)}
        \longrightarrow \rho^{-3/2},
        \qquad t\to\infty,
\end{equation}
not \(\rho^{-2\alpha}=\rho^{-1}\).
\end{proposition}

\begin{proof}
The moments are
\[
        S_k=1-\frac{2}{2^k}=1-2^{1-k},
\]
so \(S_1=0\), \(S_2=1/2\), and \(S_3=3/4\). Moreover,
\[
        c_2(\alpha)=-\frac1{\Gamma(1-2\alpha)}
\]
is negative for \(0<\alpha<1/2\) and positive for \(1/2<\alpha<1\). Formula \eqref{eq:dM-first-nonzero} therefore gives, at each fixed nonexceptional order,
\[
        \partial_\alpha M_\alpha(t)
        \sim -2S_2c_2(\alpha)t^{-2\alpha}\log t
        =-c_2(\alpha)t^{-2\alpha}\log t.
\]
This proves \eqref{eq:sharp-opposite-derivatives}. A continuous injective function on an interval is strictly monotone; a differentiable increasing function cannot have a negative derivative, and a differentiable decreasing function cannot have a positive derivative. Thus the opposite derivative signs rule out injectivity on any interval crossing \(1/2\).

At \(\alpha=1/2\), the entire continuation of the reciprocal gamma function gives \(c_2(1/2)=0\), while
\[
        c_3(1/2)=\frac1{\Gamma(-1/2)}=-\frac1{2\sqrt\pi}.
\]
Taking three terms in Theorem \ref{thm:obs-expansion} gives
\[
        M_{1/2}(t)
        =S_3c_3(1/2)t^{-3/2}+O(t^{-2}),
\]
which is \eqref{eq:sharp-half-asymptotic}; division at \(\rho t\) and \(t\) yields \eqref{eq:sharp-half-ratio}.
\end{proof}

Proposition \ref{prop:sharp-gamma} shows that the excluded reciprocal-gamma orders are not merely an artifact of the proof. At such an order the first nonzero resolvent moment need not determine the first nonzero asymptotic term, and both monotonicity and the ratio exponent can change.

\subsection{The late-time qualifier and finite-time zeros}

\begin{example}[A one-mode reversal]\label{ex:sharp-early-time}
Fix \(\lambda>0\) and \(\alpha_*\in(0,1)\). The power series of the Mittag-Leffler function gives, as \(t\downarrow0\),
\[
        E_{\alpha,1}(-\lambda t^\alpha)
        =1-\frac{\lambda t^\alpha}{\Gamma(1+\alpha)}+O(t^{2\alpha}).
\]
Differentiating at the fixed order \(\alpha_*\) gives
\[
\begin{aligned}
        \left.\partial_\alpha E_{\alpha,1}(-\lambda t^\alpha)
        \right|_{\alpha=\alpha_*}
        &=-\frac{\lambda t^{\alpha_*}}{\Gamma(1+\alpha_*)}
        \bigl(\log t-\psi(1+\alpha_*)\bigr)\\
        &\quad+O(t^{2\alpha_*}|\log t|).
\end{aligned}
\]
The derivative is positive for all sufficiently small \(t\), whereas Corollary \ref{cor:one-mode} makes it negative for all sufficiently large \(t\). Thus one-mode monotonicity is genuinely eventual rather than an all-time property.
\end{example}

\begin{example}[The ratio can be undefined at a finite time]\label{ex:sharp-finite-zero}
Fix \(\alpha_*\in(0,1)\). Choose \(t_*>0\) sufficiently small and put
\[
        p_*=E_{\alpha_*,1}(-t_*^{\alpha_*}),
        \qquad
        q_*=E_{\alpha_*,1}(-2t_*^{\alpha_*}).
\]
Since \(p_*,q_*\to1\) as \(t_*\downarrow0\), we may require \(q_*>p_*/2\). With
\[
        \cA=\mathrm{diag}(1,2),
        \qquad
        a_1=q_*,
        \qquad
        a_2=-p_*,
\]
one has
\[
        M_{\alpha_*}(t_*)=q_*p_*-p_*q_*=0,
        \qquad
        S_1=q_*-\frac{p_*}{2}>0.
\]
Thus the first-moment hypotheses hold, but the denominator of the ratio vanishes at the prescribed finite time. Theorem \ref{thm:ratio} correctly asserts nonvanishing only after a configuration-dependent late-time threshold.
\end{example}

\subsection{Stationary modes, the cancellation index, and noise}

\begin{example}[An unsubtracted stationary mode]\label{ex:sharp-kernel}
If strict positivity is removed, let
\[
        \cA=\mathrm{diag}(0,1),
        \qquad
        M_\alpha(t)=1+E_{\alpha,1}(-t^\alpha).
\]
Then
\[
        \frac{M_\alpha(\rho t)}{M_\alpha(t)}\longrightarrow1,
\]
not \(\rho^{-\alpha}\). After subtracting the stationary component, the corrected observation is \(E_{\alpha,1}(-t^\alpha)\), whose ratio does tend to \(\rho^{-\alpha}\). This is the reason for the stationary-mode qualification following Assumption \ref{ass:A}.
\end{example}

The leading two-time ratio identifies only the product \(m\alpha\). For example, \((m,\alpha)=(1,0.6)\) and \((m,\alpha)=(2,0.3)\) both produce \(\rho^{-0.6}\). Thus the requirement in Remark \ref{rem:m-known} that \(m\) be computed from the known forward configuration is substantive; \(m\) and \(\alpha\) cannot in general be separated from the leading ratio alone.

\begin{proposition}[Failure under fixed additive noise]\label{prop:sharp-additive-noise}
Assume the hypotheses of Theorem \ref{thm:ratio}, fix the true order \(\alpha\in I\), and let
\[
        d_1=M_\alpha(t),
        \qquad
        d_2=M_\alpha(\rho t).
\]
For any fixed \(\delta>0\), define noisy data
\(\widetilde d_1=d_1+\delta\) and \(\widetilde d_2=d_2+\delta\). Then
\[
        -\frac1{m\log\rho}
        \log\left|\frac{\widetilde d_2}{\widetilde d_1}\right|
        \longrightarrow0,
        \qquad t\to\infty.
\]
Hence no time-uniform consistency statement for the log-ratio estimator is possible under fixed nonzero additive noise.
\end{proposition}

\begin{proof}
The asymptotic expansion of Theorem \ref{thm:ratio} implies \(d_1,d_2\to0\). Therefore \(\widetilde d_1,\widetilde d_2\to\delta\), their quotient tends to \(1\), and the displayed logarithm tends to zero.
\end{proof}

The restriction \(\eta<1\) in Proposition \ref{prop:relative-noise} guarantees that the logarithmic estimator is defined for every admissible perturbation satisfying \(\abs{\varepsilon_j}\le\eta\). For an individual realization, it is enough that \(1+\varepsilon_j>0\) for \(j=1,2\).

\section{Application to anomalous heat conduction in a finite rod}\label{sec:rod}

We first fix the scaling convention. Let \(t_{\rm ph}\) and \(x_{\rm ph}\) denote physical time and position, choose reference scales \(\tau_*>0\) and \(\ell_*>0\), and consider the dimensionally consistent physical model
\[
        \tau_*^{\alpha-1}\partial_{t_{\rm ph}}^\alpha u
        -\kappa_{\rm ph}u_{x_{\rm ph}x_{\rm ph}}
        +q_{\rm ph}u=0.
\]
With
\[
        t=\frac{t_{\rm ph}}{\tau_*},\qquad
        x=\frac{x_{\rm ph}}{\ell_*},\qquad
        \kappa=\frac{\kappa_{\rm ph}\tau_*}{\ell_*^2},\qquad
        q=q_{\rm ph}\tau_*,
\]
and with \(L\) denoting the nondimensional rod length, the equation takes the form below. In particular, every occurrence of \(\log t\) means \(\log(t_{\rm ph}/\tau_*)\).

Let \(L>0\), \(\kappa>0\), and \(q\ge0\). Consider the nondimensional one-dimensional model
\begin{equation}\label{eq:rod-pde}
\begin{cases}
        \partial_t^\alpha u(t,x)-\kappa u_{xx}(t,x)+q u(t,x)=0,
            & t>0,\ 0<x<L,\\
        u(t,0)=u(t,L)=0, & t>0,\\
        u(0,x)=\varphi(x), & 0<x<L.
\end{cases}
\end{equation}
Here \(\alpha\) is the unknown memory parameter, while \(\kappa\) and \(q\) are known nondimensional spatial and leakage coefficients obtained from the chosen reference scales. The spatial operator is
\begin{equation}\label{eq:rod-operator}
        \cA=-\kappa\frac{\dd^2}{\dd x^2}+q,
        \qquad
        \Dom(\cA)=H^2(0,L)\cap H_0^1(0,L).
\end{equation}
Its eigenpairs are
\begin{equation}\label{eq:rod-spectrum}
        \lambda_n=\kappa\left(\frac{n\pi}{L}\right)^2+q,
        \qquad
        v_n(x)=\sqrt{\frac{2}{L}}\sin\left(\frac{n\pi x}{L}\right).
\end{equation}
Thus, for \(\varphi\in L^2(0,L)\),
\begin{equation}\label{eq:rod-solution}
        u_\alpha(t,x)=\sum_{n=1}^\infty
        \varphi_nE_{\alpha,1}(-\lambda_nt^\alpha)v_n(x),
        \qquad
        \varphi_n=\int_0^L\varphi(x)v_n(x)\dd x.
\end{equation}
For every \(t>0\), the solution belongs to \(\Dom(\cA)\) and hence has a continuous representative. Indeed, the standard estimate \(rE_{\alpha,1}(-r)\le C_I\), valid for \(r\ge0\) and \(\alpha\in I\Subset(0,1)\), implies \(\norm{\cA u_\alpha(t)}_{L^2(0,L)}\le C_I t^{-\alpha}\norm{\varphi}_{L^2(0,L)}\); see \cite{Pollard1948,GorenfloKilbasMainardiRogosin2020}. Point observations \(u_\alpha(t,x_*)\), \(x_*\in(0,L)\), are therefore meaningful.

For a fixed sensor \(x_*\in(0,L)\), the coefficients in \eqref{eq:M-def} are
\begin{equation}\label{eq:point-coefficients}
        a_n=\varphi_nv_n(x_*).
\end{equation}
They satisfy Assumption \ref{ass:moment-finite}. Indeed,
\begin{equation}\label{eq:point-summability}
        \sum_{n=1}^\infty\frac{\abs{\varphi_nv_n(x_*)}}{\lambda_n}
        \le
        \left(\sum_{n=1}^\infty\abs{\varphi_n}^2\right)^{1/2}
        \left(\sum_{n=1}^\infty\frac{\abs{v_n(x_*)}^2}{\lambda_n^2}\right)^{1/2}<\infty,
\end{equation}
because \(\lambda_n\asymp n^2\) and \(\abs{v_n(x_*)}\le\sqrt{2/L}\).

The first moment has a simple elliptic meaning. Let
\begin{equation}\label{eq:w-def}
        w=\cA^{-1}\varphi.
\end{equation}
Then \(w\in H^2(0,L)\cap H_0^1(0,L)\), and
\begin{equation}\label{eq:S1-point}
        S_1(x_*)=\sum_{n=1}^\infty
        \frac{\varphi_nv_n(x_*)}{\lambda_n}=w(x_*).
\end{equation}
Equivalently, \(w\) is the solution of
\begin{equation}\label{eq:elliptic-resolvent-rod}
\begin{cases}
        -\kappa w''(x)+q w(x)=\varphi(x), & 0<x<L,\\
        w(0)=w(L)=0.
\end{cases}
\end{equation}

\begin{theorem}[Point-sensor order identification in a rod]\label{thm:rod-point}
Let \(I=[\alpha_0,\alpha_1]\Subset(0,1)\), \(\varphi\in L^2(0,L)\), and \(x_*\in(0,L)\). If
\begin{equation}\label{eq:rod-nondegenerate}
        (\cA^{-1}\varphi)(x_*)\ne0,
\end{equation}
then there exists \(T_I>1\) such that, for every \(t_0\ge T_I\), the equation
\begin{equation}\label{eq:rod-one-measurement}
        u_\alpha(t_0,x_*)=d_0
\end{equation}
has at most one solution \(\alpha\in I\). The map \(\alpha\mapsto u_\alpha(t_0,x_*)\) is decreasing for large \(t_0\) if \((\cA^{-1}\varphi)(x_*)>0\), and increasing if \((\cA^{-1}\varphi)(x_*)<0\).
\end{theorem}

\begin{proof}
For the point observation, \eqref{eq:point-summability} verifies Assumption \ref{ass:moment-finite}, and \eqref{eq:S1-point} identifies \(S_1\) with \((\cA^{-1}\varphi)(x_*)\). The conclusion follows from Corollary \ref{cor:S1}.
\end{proof}

The higher-moment cancellation theorem also has a pointwise version. Put \(w_m=\cA^{-m}\varphi\). Then \(w_m\in\Dom(\cA^m)\subset\Dom(\cA)\), and its spectral partial sums converge in the graph norm of \(\cA^m\), hence also in the graph norm of \(\cA\). Elliptic regularity makes the latter norm equivalent to the \(H^2(0,L)\)-norm on \(H^2(0,L)\cap H_0^1(0,L)\). Thus the expansion converges in \(H^2(0,L)\) and, by one-dimensional Sobolev embedding, uniformly on \([0,L]\). Therefore the value at \(x_*\) is meaningful, and
\[
        S_m(x_*)=\sum_{n=1}^\infty
        \frac{\varphi_nv_n(x_*)}{\lambda_n^m}
        =(\cA^{-m}\varphi)(x_*).
\]

\begin{corollary}[Point-sensor identification after cancellation]\label{cor:rod-higher-moment}
Let \(I=[\alpha_0,\alpha_1]\Subset(0,1)\), \(\varphi\in L^2(0,L)\), \(x_*\in(0,L)\), and \(m\in\N\). Suppose that
\begin{equation}\label{eq:rod-higher-moment-condition}
        (\cA^{-k}\varphi)(x_*)=0\quad (1\le k\le m-1),
        \qquad
        (\cA^{-m}\varphi)(x_*)\ne0,
\end{equation}
with the vanishing condition void for \(m=1\). If \(I\) satisfies \eqref{eq:no-gamma-zero}, then one sufficiently late point value \(u_\alpha(t_0,x_*)\) determines at most one \(\alpha\in I\). In addition, for every \(\rho>1\),
\[
        \frac{u_\alpha(\rho t,x_*)}{u_\alpha(t,x_*)}
        =\rho^{-m\alpha}(1+O(t^{-\alpha_0}))
\]
uniformly for \(\alpha\in I\), and the logarithmic stability estimate of Theorem \ref{thm:ratio} holds with this \(m\).
\end{corollary}

\begin{proof}
The preceding identity gives \(S_k(x_*)=(\cA^{-k}\varphi)(x_*)\). The assertion follows from Theorems \ref{thm:first-nonzero} and \ref{thm:ratio}.
\end{proof}

The condition \eqref{eq:rod-nondegenerate} is much less restrictive than finite-mode sign assumptions. It is automatic for physically common non-negative initial temperature profiles.

\begin{corollary}[Positive initial temperature]\label{cor:positive-rod}
Assume \(\varphi\in L^2(0,L)\), \(\varphi\ge0\) almost everywhere, and \(\varphi\not\equiv0\). Then
\begin{equation}\label{eq:positive-resolvent}
        (\cA^{-1}\varphi)(x)>0,
        \qquad 0<x<L.
\end{equation}
Consequently, for every interior sensor \(x_*\in(0,L)\), one sufficiently late point temperature value determines at most one admissible \(\alpha\in I\); for data generated by the model, the admissible order is unique. The corresponding observation map is eventually strictly decreasing.
\end{corollary}

\begin{proof}
The Dirichlet Green function of \(-\kappa\dd^2/\dd x^2+q\), \(q\ge0\), is strictly positive in \((0,L)\times(0,L)\), as also follows from the one-dimensional maximum principle \cite[Ch. 6]{Evans2010}. Hence
\[
        (\cA^{-1}\varphi)(x)=\int_0^LG(x,y)\varphi(y)\dd y>0
\]
for every \(x\in(0,L)\), because \(\varphi\ge0\) and is not identically zero. The uniqueness statement follows from Theorem \ref{thm:rod-point}.
\end{proof}

\begin{theorem}[Two-time point-sensor log-ratio estimate in a rod]\label{thm:rod-ratio}
Let the assumptions of Theorem \ref{thm:rod-point} hold and fix \(\rho>1\). Then, for all sufficiently large \(t\), the ratio
\begin{equation}\label{eq:rod-ratio}
        R_\alpha^{\rm rod}(t,\rho)=
        \frac{u_\alpha(\rho t,x_*)}{u_\alpha(t,x_*)}
\end{equation}
is well-defined and satisfies
\begin{equation}\label{eq:rod-ratio-asymp}
        R_\alpha^{\rm rod}(t,\rho)=\rho^{-\alpha}(1+O(t^{-\alpha_0}))
\end{equation}
uniformly for \(\alpha\in I\). Moreover,
\begin{equation}\label{eq:rod-ratio-stability}
        \abs{\alpha-\widetilde\alpha}
        \le \frac{2}{\log\rho}
        \left\lvert
        \log\left\lvert\frac{u_\alpha(\rho t,x_*)}{u_\alpha(t,x_*)}\right\rvert
        -
        \log\left\lvert\frac{u_{\widetilde\alpha}(\rho t,x_*)}{u_{\widetilde\alpha}(t,x_*)}\right\rvert
        \right\rvert
\end{equation}
for all \(\alpha,\widetilde\alpha\in I\).
\end{theorem}

\begin{proof}
Theorem \ref{thm:rod-point} assumes \(S_1=(\cA^{-1}\varphi)(x_*)\ne0\), so Theorem \ref{thm:ratio} applies with \(m=1\).
\end{proof}

Thus, from two measured values
\[
        d_1=u_\alpha(t,x_*),
        \qquad
        d_2=u_\alpha(\rho t,x_*),
\]
one obtains the leading estimator
\begin{equation}\label{eq:rod-estimator}
        \alpha_{\rm LR}
        =-\frac{1}{\log\rho}\log\left\lvert\frac{d_2}{d_1}\right\rvert.
\end{equation}
This formula is not meant to replace solving the exact nonlinear equation; rather, it provides a scale-free initialization and a transparent interpretation of the relative data.

\begin{example}[Uniform initial temperature, no leakage]\label{ex:constant-temp}
Let \(q=0\) and \(\varphi(x)=\Theta_0\), where \(\Theta_0>0\). Then \(w=\cA^{-1}\varphi\) solves
\[
        -\kappa w''(x)=\Theta_0,
        \qquad w(0)=w(L)=0,
\]
and therefore
\begin{equation}\label{eq:constant-w}
        w(x)=\frac{\Theta_0}{2\kappa}x(L-x)>0,
        \qquad 0<x<L.
\end{equation}
Hence every interior temperature sensor satisfies the non-degeneracy condition \eqref{eq:rod-nondegenerate}. One late reading gives uniqueness on \(I\), while two late readings give the log-ratio approximation \eqref{eq:rod-estimator}.
\end{example}

\begin{example}[Cancellation at the first moment]\label{ex:cancellation}
It is possible that \((\cA^{-1}\varphi)(x_*)=0\), for instance when signed initial data are arranged so that positive and negative contributions cancel at the sensor. The theory does not necessarily fail. If
\[
        (\cA^{-1}\varphi)(x_*)=0,
        \qquad
        (\cA^{-2}\varphi)(x_*)\ne0,
\]
and the admissible interval \(I\) avoids \(\alpha=1/2\), then Corollary \ref{cor:rod-higher-moment} applies with \(m=2\). A single sufficiently late point measurement is again unique, and the two-time ratio behaves as \(\rho^{-2\alpha}\). This illustrates why higher resolvent moments are useful: cancellation of the first moment is not by itself an obstruction to order identification.
\end{example}

\section{Sharpness and counterexamples for point sensors}\label{sec:sharpness-rod}

The point-sensor assumptions in Section \ref{sec:rod} are likewise structural rather than cosmetic.

\begin{example}[A sensor at a nodal point]\label{ex:sharp-nodal-sensor}
Let \(L=\kappa=1\), \(q=0\), \(\varphi=v_2=\sqrt2\sin(2\pi x)\), and place the sensor at \(x_*=1/2\). Then
\[
        u_\alpha(t,x)
        =E_{\alpha,1}(-4\pi^2t^\alpha)v_2(x),
\]
so
\[
        u_\alpha(t,1/2)=0
\]
for every \(t>0\) and every \(\alpha\in(0,1)\). In fact,
\[
        (\cA^{-k}\varphi)(1/2)
        =(4\pi^2)^{-k}v_2(1/2)=0,
        \qquad k\ge1.
\]
Thus no order can be identified from this sensor, and the nondegeneracy condition in Theorem \ref{thm:rod-point}, or a higher-moment replacement, cannot simply be omitted. Similarly, a boundary sensor at \(x_*=0\) or \(x_*=L\) records zero for every Dirichlet solution, which explains the restriction to interior sensors.
\end{example}

\begin{example}[Loss of strict positivity]\label{ex:sharp-rod-kernel}
If the assumption \(q\ge0\) is removed, take \(L=\kappa=1\) and \(q=-\pi^2\). Then the first Dirichlet eigenvalue of
\(\cA=-\dd^2/\dd x^2-\pi^2\) is zero. For \(\varphi=v_1=\sqrt2\sin(\pi x)\),
\[
        u_\alpha(t,x)=v_1(x)
\]
for all \(t\) and all \(\alpha\). Every interior sensor away from the nodes sees a nonzero signal that is completely independent of the fractional order. This is the rod counterpart of Example \ref{ex:sharp-kernel}.
\end{example}

\begin{remark}[Why the arbitrary-\(L^2\) point-sensor statement is dimension dependent]\label{rem:sharp-dimension}
The rod theorem uses the one-dimensional embedding \(H^2(0,L)\hookrightarrow C[0,L]\). It cannot be transferred unchanged to arbitrary \(L^2\)-data in all dimensions. For example, let \(\Omega=B_1(0)\subset\R^d\) with \(d\ge5\), let \(\cA=-\Delta+I\) with Dirichlet boundary condition, and choose \(\chi\in C_c^\infty(\Omega)\) equal to one near the origin. If
\[
        0<\beta<\frac{d-4}{2},
        \qquad
        w(x)=\chi(x)|x|^{-\beta},
\]
then \(w\in H^2(\Omega)\cap H_0^1(\Omega)\) but \(w\) is unbounded at the origin. Setting \(\varphi=\cA w\in L^2(\Omega)\) gives \(\cA^{-1}\varphi=w\), whose value at the proposed sensor is not finite. Additional regularity, averaging, or a lower-dimensional setting is therefore required for point observations of arbitrary \(L^2\)-initial data.
\end{remark}

\section{Computational illustration of the log-ratio estimator}\label{sec:numerical}

This section is included only to illustrate the inverse reconstruction suggested by Theorem \ref{thm:rod-ratio}; the proofs above are independent of the computation. Take
\[
        L=1,\qquad \kappa=1,\qquad q=0,
        \qquad \varphi(x)=1,
        \qquad x_*=\frac12,
\]
and use two point measurements with \(\rho=2\). The normalized eigenfunctions are
\[
        v_n(x)=\sqrt{2}\sin(n\pi x),
        \qquad \lambda_n=n^2\pi^2 .
\]
For this data set only odd modes contribute at the midpoint, and
\[
        a_n=\varphi_n v_n(1/2)
        =\frac{4}{n\pi}(-1)^{(n-1)/2},
        \qquad n\text{ odd}.
\]
Moreover \((\cA^{-1}\varphi)(1/2)=1/8\), so the first resolvent moment is non-zero and the leading ratio is \(2^{-\alpha}\).

The synthetic data in Table \ref{tab:ratio-numerics} were obtained by truncating the spectral series to odd modes \(n\le121\). The accompanying script \texttt{reproduce\_log\_ratio\_table.py} evaluates the completely monotone representation
\[
E_{\alpha,1}(-z)=\int_0^\infty e^{-r z^{1/\alpha}}
        \frac{1}{\pi}
        \frac{r^{\alpha-1}\sin(\pi\alpha)}{r^{2\alpha}+2r^\alpha\cos(\pi\alpha)+1}\,\dd r,
        \qquad z>0,
\]
through the equivalent smoothed form
\[
E_{\alpha,1}(-z)=\frac{\sin(\pi\alpha)}{\pi\alpha z}
\int_0^\infty
\frac{e^{-s^{1/\alpha}}}{1+2s\cos(\pi\alpha)/z+s^2/z^2}\,\dd s.
\]
The script uses \texttt{scipy.integrate.quad} on \([0,\infty)\) with absolute and relative tolerances \(2\times10^{-13}\) and subdivision limit \(200\). It writes the table values to a CSV file and optionally repeats the computation with odd modes \(n\le241\) as a modal-refinement check. No modes beyond the stated truncation are included in the displayed point temperatures. The reported reconstruction is
\[
        \alpha_{\rm LR}(t)
        =-\frac{1}{\log 2}
        \log\left|\frac{u_\alpha(2t,1/2)}{u_\alpha(t,1/2)}\right| .
\]

\begin{table}[ht]
\centering
\caption{Two-time reconstruction from the odd-mode truncation \(n\le121\) for the rod model.}
\label{tab:ratio-numerics}
\begin{tabular}{ccccc}
\toprule
True order & \(t\) & \(u_\alpha(t,1/2)\) & \(u_\alpha(2t,1/2)/u_\alpha(t,1/2)\) & \(\alpha_{\rm LR}(t)\) \\
\midrule
0.35 & 20  & \(3.10979\times10^{-2}\) & 0.787509 & 0.3446 \\
0.35 & 80  & \(1.92703\times10^{-2}\) & 0.786369 & 0.3467 \\
0.35 & 320 & \(1.19104\times10^{-2}\) & 0.785677 & 0.3480 \\
\addlinespace
0.60 & 20  & \(9.39932\times10^{-3}\) & 0.658330 & 0.6031 \\
0.60 & 80  & \(4.07658\times10^{-3}\) & 0.659121 & 0.6014 \\
0.60 & 320 & \(1.77160\times10^{-3}\) & 0.659476 & 0.6006 \\
\addlinespace
0.80 & 20  & \(2.50809\times10^{-3}\) & 0.571446 & 0.8073 \\
0.80 & 80  & \(8.20798\times10^{-4}\) & 0.573397 & 0.8024 \\
0.80 & 320 & \(2.70056\times10^{-4}\) & 0.574036 & 0.8008 \\
\bottomrule
\end{tabular}
\end{table}

The point values in the table are truncated-series outputs rounded to six significant digits; they are not presented as certified full-series values to all displayed digits. The ratios and four-decimal estimators are unchanged at the displayed precision under further modal refinement. The table shows the behavior predicted by the theory: as the observation time increases, the quotient approaches \(2^{-\alpha}\) and the log-ratio estimate approaches the true fractional order. In practical inverse computations, \(\alpha_{\rm LR}\) can be used as a scale-free initial guess for solving the nonlinear equation based on a more accurate forward model.

\section{Conclusion}\label{sec:conclusion}

We proved that late-time scalar observations of time-Caputo diffusion are organized by resolvent moments. A moment-completeness argument shows that every nontrivial grouped spectral observation satisfying the summability assumption has a finite cancellation index \(m\). The first nonzero moment \(S_m\) determines the eventual sign of \(\partial_\alpha M_\alpha(t)\), provided the associated reciprocal-gamma coefficient does not vanish on the admissible interval. This yields uniqueness from one sufficiently late scalar measurement and covers exact cancellation beyond the standard condition \(S_1\ne0\).

For bounded Hilbert-space observations, \(S_m=\inner{\cA^{-m}\varphi}{h}\), giving an intrinsic operator-theoretic criterion. For one-dimensional heat conduction, point observations are covered for arbitrary \(L^2\)-initial data, and the leading condition is \((\cA^{-1}\varphi)(x_*)\ne0\). Positivity of the Dirichlet Green function makes this condition automatic for nonnegative, nonzero initial temperature at every interior sensor. The near-cancellation remark also shows why the eventual index may be difficult to observe over a finite time window when a lower moment is small but not exactly zero.

The two-time ratio removes the leading amplitude from the data variable. If \(S_m\) is the first nonzero moment, then \(M_\alpha(\rho t)/M_\alpha(t)\sim\rho^{-m\alpha}\). The integer \(m\) is computed from the known forward configuration. Corollary \ref{cor:combined-error} combines the asymptotic bias with a time-uniform relative-noise term, while fixed additive noise remains limited by the algebraically decaying signal amplitude. For \(m=1\), the estimator is a discrete logarithmic-slope version of earlier late-time reconstruction formulas; the new feature is its extension to signed observations and arbitrary finite cancellation index.

The sharpness sections show that the principal qualifications are intrinsic. A two-mode admissible observation loses monotonicity across the reciprocal-gamma zero at \(\alpha=1/2\), and its ratio exponent changes at the exceptional order. Further examples show that monotonicity is only eventual, finite-time denominators can vanish, stationary modes must be removed, fixed additive noise defeats the late-time ratio, and point sensors require both nondegeneracy and sufficient spatial regularity.

Possible extensions include simultaneous recovery of \(\alpha\) and uncertain forward coefficients, statistical selection of observation times under an additive noise floor, and analogous resolvent-moment criteria for multi-term or distributed-order fractional diffusion.

\section*{Acknowledgments and declarations}

The author gratefully acknowledges the support of the Science Committee of the Ministry of Science and Higher Education of the Republic of Kazakhstan, Grant No. AP23487589.

\medskip
\noindent\textbf{Competing interests.} The author declares that there are no known competing financial interests or personal relationships that could have appeared to influence the work reported in this paper.

\medskip
\noindent\textbf{Data and code availability.} No external data sets were used. The companion script \texttt{reproduce\_log\_ratio\_table.py}, included as ancillary material, generates the synthetic values in Section \ref{sec:numerical}, writes them to CSV, and includes an optional modal-refinement check.

\end{document}